\documentclass{siamart220329}

\usepackage{lipsum}
\usepackage{amsfonts}
\usepackage{graphicx}
\usepackage{epstopdf}
\usepackage{algorithmic}
\usepackage{amsmath,amssymb}

\newsiamremark{remark}{Remark}
\newsiamremark{hypothesis}{Hypothesis}
\crefname{hypothesis}{Hypothesis}{Hypotheses}

\headers{Analysis of EIM and CGA}{Y. Li}

\title{A new analysis of empirical interpolation methods and Chebyshev greedy algorithms\thanks{This work was partially supported by the National Science Foundation of China (no.~12471346) and   the Fundamental Research Funds for the Central Universities (no.~226-2023-00039).}}

\author{
  Yuwen Li\thanks{School of Mathematical Sciences, Zhejiang University, Hangzhou, Zhejiang 310058, China 
  (\email{liyuwen@zju.edu.cn}).}}

\usepackage{amsopn}
\numberwithin{theorem}{section}
\numberwithin{equation}{section}
\numberwithin{algorithm}{section}

\begin{document}

\maketitle

\begin{abstract}
We present new convergence estimates of generalized empirical interpolation methods in terms of the entropy numbers of the parametrized function class. Our analysis is transparent and leads to sharper convergence rates than the classical analysis via the Kolmogorov $n$-width. In addition, we also derive novel entropy-based convergence estimates of the Chebyshev greedy algorithm for sparse $n$-term nonlinear approximation of a target function. This also improves classical convergence analysis when corresponding entropy numbers decay fast enough.
\end{abstract}

\begin{keywords}
Empirical interpolation method, reduced basis greedy algorithm, Chebyshev greedy algorithm, metric entropy numbers, parametrized PDE, dictionary approximation.
\end{keywords}

\begin{MSCcodes}
41A46, 41A65, 65J05, 65M12 
\end{MSCcodes}

\section{Introduction}
Parametrized Partial Differential Equations (PDEs) are important mathematical models arising in a variety of physical contexts, e.g, material design, shape optimization, uncertainty quantification, inverse problems etc. For these real-world applications, it is necessary to efficiently evaluate the query of PDE solutions for many instances of the parameter, such as a family of elasticity and diffusion coefficients in deterministic PDEs and different realizations of random fields in stochastic PDEs. In particular, the numerical solver for a parametrized family of PDEs is expected to be much faster than repeatedly running  general-purpose PDE solvers, e.g., finite element and finite difference methods, for each parameter instance.

For many-query tasks such as solving parametric PDEs, model reduction is a ubiquitous strategy that  replaces the large-scale high-fidelity PDE model with an easy-to-solve reduced model. In recent decades, the Reduced Basis Method (RBM) is possibly one of the most notable model reduction techniques for parametric PDEs, see, e.g., \cite{Maday2006,CohenDeVore2015,Quarteroni2016} and references therein. Interested readers are also referred to \cite{KunischVolkwein2002,Nouy2010,BennerGugercinWillcox2015,SanBorggaard2015,KutzBruntonProctor2016,LeeCarlberg2020,AntoulasBeattieGugercin2020} and references therein for other reduced order modeling of  dynamical systems. However, the efficiency of RBMs and many other projection-based model reduction methods hinges on the affine parametrization structure  which might not be available in practice, see Section \ref{subsecRBGA} for more details. 

To overcome this drawback, the Empirical Interpolation Method (EIM) was invented in \cite{BarraultMadayNguyenPatera2004} for constructing affinely parametric approximants to coefficient functions in parametrized PDEs. Since then, the EIM was generalized in many ways and became an extremely popular tool for obtaining the affine structure as well as for efficient reduced order modeling and analysis of complex nonlinear systems and big datasets, see, e.g., \cite{NguyenPateraPeraire2008,EftangGreplPatera2010,ChaturantabutSorensen2010,EftangGreplPatera2012,MadayMula2013,NegriManzoniAmsallem2015,DrmacGugercin2016,SorensenEmbree2016,Saibaba2020,PeherstorferDrmacGugercin2020}. Classical convergence analysis of  EIMs could be found in, e.g., \cite{NguyenPateraPeraire2008,ChaturantabutSorensen2012,EftangGreplPatera2013,MadayMulaTurinici2016}. In fact, the EIM can be viewed as a weak reduced basis greedy algorithm with varying threshold constants at each iteration. Therefore, the work \cite{MadayMulaTurinici2016} modified the error analysis  of greedy RBMs in \cite{DeVorePetrova2013} and derived convergence rates of the EIM in terms of the Kolmogorov $n$-width $d_n(K)$ of the compact set $K$ of parametrized functions. For example, Corollary 14 in \cite{MadayMulaTurinici2016} shows that for each $t>0$,
\begin{equation}\label{classicalEIM}
    d_n(K)\leq Cn^{-t}\Longrightarrow\sup_{f\in K}\|f-\Pi_nf\|\leq C(1+\Lambda_n)^3n^{-t+1},
\end{equation}
where $\Pi_n$ is the EIM interpolation at the $n$-th iteration, $C$ is a generic constant that may change from line to line but is independent of $n$ throughout this paper, and $\Lambda_n$ is the norm of $\Pi_n$ and is called the Lebesgue constant. The analysis in \cite{MadayMulaTurinici2016} assumes that $\{\Lambda_n\}_{n\geq1}$ is a monotonically increasing sequence.

In this paper, we shall investigate the convergence of EIMs in a way  different from \cite{MadayMulaTurinici2016}. Instead of the Kolmogorov $n$-width, the key concept in our analysis are the metric entropy numbers $\varepsilon_n({\rm co}(K))$ of the symmetric convex hull ${\rm co}(K)$ of the underlying function set $K$. Such an idea originates from the author's work \cite{LiSiegel2024} about entropy-based convergence rates of greedy RBMs. 
Our new error analysis of EIMs is transparent and sharper than classical results of \cite{MadayMulaTurinici2016} in many cases.  In a Banach space, the proposed  convergence estimate is
\begin{equation}\label{ourEIMexample}
        \sup_{f\in K}\|f-\Pi_{n-1}f\|\leq C(1+\Lambda_{n-1})\left(\prod_{k=1}^{n-1}(1+\Lambda_k)\right)^{\frac{1}{n}}n\varepsilon_n({\rm co}(K)).
\end{equation}
Here we do not make any assumption on $\{\Lambda_n\}_{n\geq0}$ and $\{\varepsilon_n({\rm co}(K))\}_{n\geq0}$. For a standard parametrized compact set $K$, it is shown in \cite{SiegelXuFoCM} that $n^{\frac{1}{2}}\varepsilon_n({\rm co}(K))=d_n(K)=O(n^{-t})$ as $n\to\infty$ for some $t>0$. It is also numerically observed (see \cite{MadayNguyenPateraPau2009}) that $\Lambda_n$ mildly grows in practice. Therefore, \eqref{ourEIMexample} is sharper than the classical result \eqref{classicalEIM} for a wide range of function classes. In addition, the estimate \eqref{ourEIMexample} could be improved in Hilbert spaces and special Banach spaces, see Section \ref{secEIM} for the detailed analysis and comparisons.

The EIMs are greedy algorithms designed to simultaneously approximate a compact collection of parameter-dependent functions. On the other hand, greedy algorithms are also widely used for sparse nonlinear approximation of a single target function. These algorithms adaptively select basis functions $\{g_i\}_{1\leq i\leq n}$ from a compact dictionary $K\subset X$ and use a $n$-term linear combination $$f_n=\sum_{i=1}^nc_ig_i,\quad g_i\in K$$ 
to approximate a function $f\in X$. Important applications of sparse greedy approximations include nonparametric statistical regression \cite{FriedmanStuetzle1981,Jones1992}, matching pursuit in compressed sensing \cite{MallatZhang1993}, training shallow neural networks \cite{SiegelXu2022} etc. We refer to \cite{Temlyakov2008,Temlyakov2011} for a thorough introduction to greedy dictionary approximation. Among such greedy algorithms for a single target function, it is recently shown in \cite{SiegelXu2022,LiSiegel2024} that the Orthogonal Greedy Algorithm (OGA) in Hilbert spaces exhibits the fastest convergence rate provided ${\rm co}(K)$ has small entropy numbers.

The OGA was generalized to Banach spaces and named as the Chebyshev Greedy Algorithm (CGA) by Temlyakov in \cite{Temlyakov2001}. The convergence of the CGA in its simplest form (see \cite{Temlyakov2001}) reads
\begin{equation}\label{classicalCGA}
    \|f-f_n\|\leq Cn^{-1+\frac{1}{s}},\quad f\in{\rm co}(K), 
\end{equation}
where $s\in(1,2]$ is the power in the modulus of smoothness of the underlying Banach space $X$ (see~\cite{LindenstraussTzafriri1979,Donahue1997,Temlyakov2001} and \eqref{rhoX}). The second main contribution of our work is the novel convergence estimate of the CGA based on entropy numbers: 
\begin{equation}\label{myCGA}
    \|f-f_n\|\leq C_fn^{\frac{1}{s}-\frac{1}{2}}\delta_{X,n}\varepsilon_n({\rm co}(K)),
\end{equation}
where $C_f$ is an explicit uniform constant depending on the target function $f$, and $\delta_{X,n}$ is the supremum of the Banach--Mazur distance between $\ell_2^n$ and $n$-dimensional subspaces of $X$. Our convergence estimate \eqref{myCGA} is a Banach version of the result about OGAs in \cite{LiSiegel2024}. 
 In addition, \eqref{myCGA} leads to  improved convergence rates of the CGA under the condition $\delta_{X,n}\varepsilon_n({\rm co}(K))=o(n^{-\frac{1}{2}})$, which is true for popular dictionaries in Banach spaces, e.g., the $\texttt{ReLU}_m$ dictionary in $L_p(\Omega)$ with $1<p<\infty$, see Section \ref{secCGA} for details.

The rest of this paper is organized as follows. In Section \ref{secGA}, we briefly introduce the EIM and the CGA in Banach spaces. In Section \ref{secEIM}, we present a convergence analysis of the EIM in terms of entropy numbers. Section \ref{secCGA} is devoted to an entropy-based convergence estimate of the CGA and related numerical experiments. 

\section{Greedy Algorithms}\label{secGA}
Let $K$ be a compact set in a Banach space $X$ under the norm $\|\bullet\|=\|\bullet\|_X$. By ${\rm co}(K)$ we denote the  symmetric convex hull of $K$, i.e.,
\[
\text{co}(K):=\overline{\Big\{\sum_{i}c_ig_i:   \sum_{i}|c_i|\leq1,~g_i\in K\text{ for each }i\Big\}}.
\]
The Kolmogorov $n$-width of $K$ is defined as 
\begin{equation*}
d_n(K)=d_n(K)_X:=\inf_{X_n}\sup_{f\in K}{\rm dist}(f,X_n),
\end{equation*}
where the infimum is taken over all $n$-dimensional subspaces of $X$, and 
\begin{equation*}
    {\rm dist}(f,X_n):=\inf_{g\in X_n}\|f-g\|.
\end{equation*}
Convergence of many classical greedy algorithms are analyzed in terms of $d_n(K)$, most notably the reduced basis greedy algorithm (see \cite{BinevCohenDahmenDeVore2011,DeVorePetrova2013,Wojtaszczyk2015}). As mentioned before, we shall follow an alternative way and analyze the convergence of the EIM and the CGA using the entropy numbers of ${\rm co}(K)$ (see~\cite{Lorentz1996}): \[
\varepsilon_n({\rm co}(K))=\varepsilon_n({\rm co}(K))_X:=\inf\{\varepsilon>0: {\rm co}(K)\text{ is covered by }2^n\text{ balls of radius } \varepsilon\}.
\]
The sequence $\{\varepsilon_n({\rm co}(K))\}_{n\geq0}$ describes the massiveness of ${\rm co}(K).$ 
It is well-known that $\lim_{n\to\infty}d_n(K)=0$ and $\lim_{n\to\infty}\varepsilon_n({\rm co}(K))=0$ for a compact $K.$ 

The metric entropy numbers is a fundamental concept in approximation, learning and probability theory \cite{Lorentz1996,Wainwright2019,SiegelXuFoCM}. For example, $\varepsilon_n({\rm co}(K))$ is the minimal distortion that one can achieve with $n$-bit encoding of ${\rm co}(K)$. Any numerical method  approximating each element of ${\rm co}(K)$ up to accuracy $\varepsilon_n({\rm co}(K))$ would require at least $n$ operations, and thus the entropy numbers  serve as a barrier for stable numerical convergence (cf.~\cite{DeVore2014,CohenDeVorePetrova2022}).  There are very useful comparisons between the decay rates of the entropy numbers and the Kolmogorov width or stable manifold width, see, e.g., Carl-type inequalities in  \cite{Carl1981,CohenDeVorePetrova2022}. Meanwhile, estimates of entropy numbers of important function sets have been established in, e.g., \cite{Lorentz1996,SiegelXuFoCM,LiLi2024REIM}. Therefore, we are motivated to develop error bounds of greedy algorithms in terms of entropy numbers.

\subsection{Reduced Basis Greedy Algorithm}\label{subsecRBGA}
In the analysis of RBMs, $K$ is the solution manifold of a parametrized PDE. Let $\mathcal{P}\subset\mathbb{R}^d$ be a compact set of parameters. For simplicity of presentation, we consider the boundary value problem for $\mu\in\mathcal{P}$:
\begin{subequations}\label{elliptic}
    \begin{align}
    -\nabla\cdot(a_\mu\nabla u_\mu)+b_\mu u_\mu&=F\quad\text{ in }\Omega,\\
    u_\mu&=0\quad\text{ on }\partial\Omega,
\end{align}
\end{subequations}
where $\Omega\subset\mathbb{R}^q$ is the physical domain, $F\in H^{-1}(\Omega)$, and 
\begin{align*}
    &0<\inf_{\mu\in\mathcal{P}}\inf_{\Omega}a_\mu\leq\sup_{\mu\in\mathcal{P}}\sup_\Omega a_\mu<\infty,\\
    &0<\inf_{\mu\in\mathcal{P}}\inf_{\Omega}b_\mu\leq\sup_{\mu\in\mathcal{P}}\sup_\Omega b_\mu<\infty.
\end{align*}
In this case, the solution manifold 
\begin{equation*}
    K=\{u_\mu\in H_0^1(\Omega): \mu\in\mathcal{P}\}
\end{equation*}
is also compact,
and the corresponding RBM approximately solves \eqref{elliptic} for many instances of the parameter $\mu$ in a low-dimensional subspace.
The abstract greedy  RBM is formulated in Algorithm \ref{WGA}.
\begin{algorithm}
\caption{Weak Reduced Basis Greedy Algorithm}\label{WGA}

Input: a compact set $K\subset X$, an integer $N$, and a sequence $\{\alpha_n\}\subset (0,1]$;

Initialization: set $X_0=\{0\}$;

\textbf{For}{$~n=1:N$}

$\qquad$Step 1: select $f_n\in K$ such that \begin{equation*}
    {\rm dist}(f_n,X_{n-1})\geq\alpha_n\sup_{f\in K}{\rm dist}(f,X_{n-1});
\end{equation*}
$\qquad$Step 2: set $X_n={\rm span}\{f_1,\ldots,f_n\}$; 

\textbf{EndFor}

Output: the subspace $X_N$ that approximates $K$.
\end{algorithm}
In RBMs for parametrized PDEs, the threshold sequence $\{\alpha_n\}$ in Algorithm \ref{WGA} is indeed constant, i.e., $\alpha_1=\cdots=\alpha_N$. The output $X_N$ is a low-dimensional subspace over which Galerkin solutions for many instances of the parameter are computed during the online stage of the RBM.
We refer to \cite{RozzaHuynhPatera2008,HesthavenRozzaStamm2016,BinevCohenDahmenDeVore2011,DeVorePetrova2013,LiSiegel2024,LiZikatanovZuo2024SISC} for more details about the implementation and analysis of RBMs. 

The efficiency of RBMs  relies on the affine structure of coefficients, e.g., when $a_\mu, b_\mu$ in \eqref{elliptic} are of the special form
\begin{align*}
    &a_\mu(x)=a_0(x)+\sum_{j=1}^{J}\theta_j(\mu)a_j(x),\\
    &b_\mu(x)=b_0(x)+\sum_{m=1}^M\omega_m(\mu)b_m(x).
\end{align*}
Such affine decomposition of $a_\mu, b_\mu$ is not necessarily available in real-world PDE models. To improve the applicability of RBMs, the works \cite{BarraultMadayNguyenPatera2004,EftangGreplPatera2010,MadayMula2013,NguyenPeraire2023,Nguyen2024} etc. have developed a variety of EIMs for approximating parameter-dependent coefficients, e.g., $\{a_\mu\}_{\mu\in\mathcal{P}}$ and $\{b_\mu\}_{\mu\in\mathcal{P}}$, by affinely separable functions uniformly with respect to $\mu\in\mathcal{P}$. In this paper, we consider  the generalized EIM (see \cite{BarraultMadayNguyenPatera2004,MadayMula2013,MadayMulaTurinici2016}) given in Algorithm \ref{EIM}.

\begin{algorithm}
\caption{Generalized Empirical Interpolation Method}\label{EIM}
Input: a compact set $K\subset X$, a collection $\mathcal{L}$ of bounded linear functionals acting on $\text{span}\{K\}$, and an integer $N\geq1$;

Initialization: let $X_0=\{0\}$, $\ell_0=0\in\mathcal{L}$ and $\Pi_0: X\rightarrow X_0$ be the zero map;

\textbf{For}{$~n=1:N$}

$\qquad$Step 1: select $f_n\in K$ such that \begin{equation*}
    \|r_n\|:=\|f_n-\Pi_{n-1}f_n\|=\max_{f\in K}\|f-\Pi_{n-1}f\|;
\end{equation*}
$\qquad$Step 2: construct $X_n={\rm span}\{f_1,\ldots,f_n\}$; 

$\qquad$Step 3: select $\ell_n\in\mathcal{L}$ such that 
\begin{equation*}
|\ell_n(r_n)|=\max_{\ell\in\mathcal{L}}|\ell(r_n)|;
\end{equation*} 
$\qquad$Step 4: let $g_n:=r_n/\ell_n(r_n)$ and $B_n=(\ell_i(g_j))_{1\leq i,j\leq n}$;

$\qquad$Step 5: define the interpolation $\Pi_n: \text{span}\{K\}\rightarrow X_n$ as
\begin{align*}
    &\Pi_nf=c_1g_1+\cdots+c_ng_n,\\
    &(c_1,\ldots,c_n)^\top:=B_n^{-1}(\ell_1(f),\ldots,\ell_n(f))^\top;
\end{align*}

\textbf{EndFor}

Output: the empirical interpolation $\Pi_N: \text{span}\{K\}\rightarrow X_N$.
\end{algorithm}

Under practical assumptions (see \cite{BarraultMadayNguyenPatera2004}), Algorithm \ref{EIM} is well-defined and satisfies the following: (i) $B_n$ is a lower triangular matrix with unit diagonal entries;  (ii) $g_1, \ldots, g_n$ is a basis of $X_n$; (iii) $\ell_1, \ldots, \ell_n$ is unisolvent on $X_n$, and thus $\Pi_n^2=\Pi_n$. Step 5 ensures that $\Pi_n$ is an interpolation based on the following degrees of freedom:
\begin{equation*}
    \ell_i(\Pi_nf)=\ell_i(f),\quad i=1,2,\ldots,n,\quad f\in K.
\end{equation*}
In the original work \cite{BarraultMadayNguyenPatera2004}, $X=L_\infty(\Omega)$ and $\mathcal{L}$ corresponds to function evaluations at a set of sample points $\{x_i\}_{1\leq i\leq L}$ in the domain $\Omega$, i.e., $\ell_{x_i}(f)=f(x_i)$ for each $\ell_{x_i}\in\mathcal{L}$. In practice, $K$ is a collection of functions smoothly parametrized by a compact set of parameters. 
Introducing basis functions $(h_1,\ldots,h_n):=(g_1,\ldots,g_n)B_n^{-1},$ the EIM interpolant has the explicit form
\begin{equation}\label{explicitEIM}
    \Pi_nf=\ell_1(f)h_1+\cdots+\ell_n(f)h_n.
\end{equation}

For the problem \eqref{elliptic}, the compact set $K$ in Algorithm \ref{EIM} is \begin{equation*}
    K=\{a_\mu\in L_\infty(\Omega): \mu\in\mathcal{P}\}\text{ or }K=\{b_\mu\in L_\infty(\Omega): \mu\in\mathcal{P}\}.
\end{equation*}
In this case, the coefficients $c_1, \ldots, c_n$ in Step 5 of Algorithm \ref{EIM} depends on $\mu$. The EIM interpolants of $a_\mu$ and $b_\mu$ in \eqref{elliptic} are of the form 
\begin{align*}
&(\Pi_Na_\mu)(x)=\sum_{i=1}^Nc_i(\mu)a_{\mu_i}(x),\\
&(\Pi_Nb_\mu)(x)=\sum_{i=1}^N\tilde{c}_i(\mu)b_{\mu_i}(x),  
\end{align*} 
where $\mu\in\mathcal{P}$ and $x\in\Omega$ are separated.
As a result, a RBM for the approximate  model
\begin{equation*}
\begin{aligned}
     -\nabla\cdot\big((\Pi_Na_\mu)\nabla \tilde{u}_\mu\big)+(\Pi_Nb_\mu)\tilde{u}_\mu&=F\quad\text{ in }\Omega,\\
     \tilde{u}_\mu&=0\quad\text{ on }\partial\Omega
\end{aligned}
\end{equation*}
is able to achieve much higher efficiency than a RBM directly applied to \eqref{elliptic}.

\subsection{Chebyshev Greedy Algorithm}
In the end of this section, we consider another class of greedy algorithms for dictionary approximation. Famous examples of these greedy algorithms include the projection pursuit in  statistical regression and the matching pursuit in compressed sensing. Among those algorithms, the OGAs (see Algorithm \ref{OGA}) in Hilbert spaces (also known as the orthogonal matching pursuit in compressed sensing) are perhaps the ones with fastest convergence rate.

\begin{algorithm}[thp]
\caption{Orthogonal Greedy Algorithm}\label{OGA}

Input: a Hilbert space $H$, a compact set $K\subset H$, a target function $f\in H,$ and an integer $N$;

Initialization: set $f_0=0\in H;$

\textbf{For}{$~n=1:N$}

$\qquad$Step 1: compute the optimizer {$$g_n=\arg\max_{g\in K}|\langle g,f-f_{n-1}\rangle_H|;$$}

$\qquad$Step 2: set $H_n={\rm span}\{g_1,\ldots,g_n\}$ and compute
\begin{equation*}
    f_n=P_{H_n}f,
\end{equation*} 
$\qquad$where  $P_{H_n}$ is the orthogonal projection onto $H_n$;

\textbf{EndFor}

Output: the iterate $f_N$ that approximates $f$.
\end{algorithm}
Convergence of the OGA have been analyzed in, e.g., \cite{BarronCohenDahmenDeVore2008,SiegelXu2022,LiSiegel2024}. In a Banach space, the CGA is a natural generalization of the OGA, see  \cite{Temlyakov2001,DereventsovTemlyakov2019}. Compared with the OGA, greedy approximation in non-Hilbert spaces like $L_p$ with $p<2$ is more robust with respect to noise, see \cite{HansonBurr1988,Donahue1997}. In \cite{Temlyakov2014}, the CGA is also a theoretical tool for proving Lebesgue-type inequalities.

For an element $g\in X$, its \emph{peak functional} $F_g$ (see \cite{Donahue1997,Temlyakov2001}) is a bounded linear functional in $X^\prime$ that satisfies
\begin{equation}\label{peak}
    \|F_g\|_{X^\prime}=1,\quad F_g(g)=\|g\|.
\end{equation}
The existence of $F_g$ is ensured by the Hahn--Banach extension theorem, see also Subsection \ref{subsec:NumCGA} for an explicit example of the peak functional $F_g$ of $g\in X=L_p(\Omega)$.
The weak version of the CGA (see \cite{Temlyakov2001}) is described in Algorithm \ref{CGA}.

\begin{algorithm}[thp]
\caption{Weak Chebyshev Greedy Algorithm}\label{CGA}

Input: a compact set $K\subset X$, a target function $f\in X$, an integer $N$, and a sequence $\{\alpha_n\}_{n\geq1}\subset(0,1]$;

Initialization: set $f_0=0\in X$ and $r_0=f-f_0=f$;

\textbf{For}{$~n=1:N$}

$\qquad$Step 1: construct a peak functional $F_{r_{n-1}}\in X^\prime$ for $r_{n-1}$;

$\qquad$Step 2: select $g_n\in K$ such that 
\begin{equation*}
|F_{r_{n-1}}(g_n)|\geq\alpha_n\sup_{g\in K}|F_{r_{n-1}}(g)|;
\end{equation*}
$\qquad$Step 3: set $X_n={\rm span}\{g_1,\ldots,g_n\}$; 

$\qquad$Step 4: select $f_n\in X_n$ such that 
\begin{equation*}
    \|r_n\|:=\|f-f_n\|=\inf_{g\in X_n}\|f-g\|;
\end{equation*}

\textbf{EndFor}

Output: the iterate $f_N$ that approximates $f$.
\end{algorithm}
When $\alpha_1=\cdots=\alpha_N=1$ and $X$ is a Hilbert space, the weak CGA reduces to the OGA in Algorithm \ref{OGA}. Rigorously speaking, Algorithms \ref{OGA} and \ref{CGA} are not implementable when the dictionary $K$ consists of infinitely many elements. In this case, $K$ would be replaced with a discrete finite subset in those greedy algorithms.


\section{Convergence of the EIM}\label{secEIM}
In this section, we derive a novel convergence estimate for the generalized EIM {(Algorithm \ref{EIM})}. We start with the norm or the Lebesgue constant of the interpolation $\Pi_n$ defined by 
\begin{equation*}
     \Lambda_n=\|\Pi_n\|_{K^\prime}:=\sup_{0\neq g\in\text{span}\{K\}}\|\Pi_ng\|/\|g\|.
\end{equation*}
For any $f\in K$, it holds that
\begin{equation}
    \inf_{g\in X_n}\|f-g\|\leq\|f-\Pi_nf\|.
\end{equation} 
On the contrary, for any $g\in X_n$ we have $\|f-\Pi_nf\|=\|(I-\Pi_n)(f-g)\|$ and obtain
\begin{equation}\label{Pinf}
        \|f-\Pi_nf\|\leq\gamma_n\inf_{g\in X_n}\|f-g\|,\quad f\in K.
\end{equation} 
The constant $\gamma_n:=\|I-\Pi_n\|_{K^\prime}$ is usually bounded as
\begin{equation}\label{gamman}
\gamma_n\leq1+\|\Pi_n\|_{K^\prime}=1+\Lambda_n.
\end{equation}
The formula \eqref{explicitEIM} implies the upper bound $\Lambda_n\leq\big(\max_{1\leq i\leq n}\|\ell_i\|_{K^\prime}\big)\sum_{j=1}^n\|h_j\|$.
Then the iterate $f_n$ in Step 2 of Algorithm \ref{EIM} satisfies 
\begin{equation}\label{EIMfn}
    \inf_{g\in X_{n-1}}\|f_n-g\|={\rm dist}(f_n,X_{n-1})\geq\gamma_{n-1}^{-1}\sup_{f\in K}{\rm dist}(f,X_{n-1}),
\end{equation}
the same as the selection criterion in the weak greedy RBM (Algorithm \ref{WGA}).

\subsection{Error Bounds of EIMs in Banach Spaces}

Let $\ell_2^n$ denote the Euclidean space $\mathbb{R}^n$ under the  $\ell_2$-norm. Our analysis utilizes the following quantity (see 
 Section III.B. of \cite{Wojtaszscyk1991})
\begin{equation*}
    \begin{aligned}
\delta_n=\delta_{X,n}:=\sup_{\substack{Y_n\subset X\\\dim Y_n=n}}\inf_{T: Y_n\rightarrow\ell_2^n}\|T\|_{Y_n\rightarrow\ell_2^n}\|T^{-1}\|_{\ell_2^n\rightarrow Y_n},
    \end{aligned}
\end{equation*}
where the infimum is taken over all isomorphism $T$ from the Banach space $Y_n$ to $\ell_2^n$. In fact, $\inf_{T: Y_n\rightarrow\ell_2^n}\|T\|_{Y_n\rightarrow\ell_2^n}\|T^{-1}\|_{\ell_2^n\rightarrow Y_n}$ is called the Banach--Mazur distance between $Y_n$ and $\ell_2^n$. Our analysis of the EIM and the CGA relies on the next fundamental 
lemma, which was developed in \cite{LiSiegel2024}.
\begin{lemma}[Lemma 3.1 from \cite{LiSiegel2024}]\label{mainlemma}
Let $K$ be a compact set in a Banach space $X$. Let $V_n=\pi^{\frac{n}{2}}/\Gamma(\frac{n}{2}+1)$ be the volume of an $\ell_2$ unit ball in $\mathbb{R}^n$. For any $v_1, \ldots, v_n\in K$ with $X_k={\rm span}\{v_1,\ldots,v_k\}$ and $X_0=\{0\}$, we have
\begin{equation*}
    \Big(\prod_{k=1}^n{\rm dist}(v_k,X_{k-1})\Big)^\frac{1}{n}\leq\delta_{X,n}(n!V_n)^{\frac{1}{n}}\varepsilon_n({\rm co}(K)).
\end{equation*}
\end{lemma}
Now we are in a position to present our convergence estimates of the generalized EIM.
\begin{theorem}\label{EIMthm}
For the weak greedy RBM (Algorithm \ref{WGA}), we have
\begin{equation*}
    \sup_{f\in K}{\rm dist}(f,X_{n-1})\leq(\alpha_1\cdots\alpha_n)^{-\frac{1}{n}}{\delta}_{X,n}(n!V_n)^\frac{1}{n}\varepsilon_n({\rm co}(K)).
\end{equation*}
In particular, the generalized EIM (Algorithm \ref{EIM}) satisfies
\begin{equation*}
    \sup_{f\in K}\|f-\Pi_{n-1}f\|\leq\gamma_{n-1}(\gamma_1\cdots\gamma_{n-1})^{\frac{1}{n}}{\delta}_{X,n}(n!V_n)^\frac{1}{n}\varepsilon_n({\rm co}(K)).
\end{equation*}
\end{theorem}
\begin{proof}
First we focus on Algorithm \ref{WGA}. It follows from Step 1 of Algorithm \ref{WGA} and Lemma \ref{mainlemma} with $v_k=f_k$ in Algorithm \ref{WGA} that 
\begin{equation*}
    \begin{aligned}
       &\sup_{f\in K}{\rm dist}(f,X_{n-1})\leq\Big(\prod_{k=1}^n\sup_{f\in K}{\rm dist}(f,X_{k-1})\Big)^\frac{1}{n}\\
       &\leq\Big(\prod_{k=1}^n\alpha_k^{-1}{\rm dist}(f_k,X_{k-1})\Big)^\frac{1}{n}\\
       &\leq(\alpha_1\cdots\alpha_n)^{-\frac{1}{n}}\delta_n(n!V_n)^\frac{1}{n}\varepsilon_n({\rm co}(K)).
    \end{aligned}
\end{equation*}
Next we note that $\{f_n\}_{n\geq1}$ in Algorithm \ref{EIM} satisfies \eqref{EIMfn}, the same selection criterion in the weak greedy RBM (Algorithm \ref{WGA}) with $\alpha_n=\gamma_{n-1}^{-1}$ and $\gamma_0=1$. Combining \eqref{Pinf}
and the first part of this theorem completes the proof. 
\end{proof}

It is easy to see that $\delta_{X,n}=1$ provided $X$ is a Hilbert space. For Banach spaces and Sobolev spaces $W^k_p(\Omega)$, it has been proven in Section III.B.9 of \cite{Wojtaszscyk1991} that  
\begin{equation}\label{deltan}
    {\delta}_{X,n}\leq
\left\{\begin{aligned}
    &\sqrt{n},\quad\text{ when $X$ is a general Banach space},\\
    &n^{|\frac{1}{2}-\frac{1}{p}|},\quad\text{when $X=W^k_p(\Omega)$ with $p\in[1,\infty]$},
\end{aligned}\right.
\end{equation}
where the convention $L_p(\Omega)=W^0_p(\Omega)$ is adopted. 
With the help of \eqref{deltan}, several important asymptotic consequences of Theorem \ref{EIMthm} are summarized in the next corollary.
\begin{corollary}\label{EIMcor}
There exists an absolute constant $C_0>0$ independent of $n$, $K$ and $X$ such that: (1) for a general Banach space $X$, we have 
\begin{equation}\label{EIMerr1}
        \sup_{f\in K}\|f-\Pi_{n-1}f\|\leq C_0(1+\Lambda_{n-1})\left(\prod_{k=1}^{n-1}(1+\Lambda_k)\right)^{\frac{1}{n}}n\varepsilon_n({\rm co}(K));
\end{equation}
(2) when $X=W^k_p(\Omega)$ with $p\in[1,\infty]$, 
it holds that
\begin{equation}\label{EIMerr2}
        \sup_{f\in K}\|f-\Pi_{n-1}f\|\leq C_0(1+\Lambda_{n-1})\left(\prod_{k=1}^{n-1}(1+\Lambda_k)\right)^{\frac{1}{n}}n^{\frac{1}{2}+|\frac{1}{2}-\frac{1}{p}|}\varepsilon_n({\rm co}(K));
\end{equation}
(3) when $X$ is a Hilbert space, it holds that
\begin{equation}\label{EIMerr3}
        \sup_{f\in K}\|f-\Pi_{n-1}f\|\leq C_0\Lambda_{n-1}\left(\Lambda_1\cdots\Lambda_{n-1}\right)^{\frac{1}{n}}n^\frac{1}{2}\varepsilon_n({\rm co}(K)).
\end{equation}
\end{corollary}
\begin{proof}
The estimates \eqref{EIMerr1} and \eqref{EIMerr2} follow from Theorem \ref{EIMthm}, the simple bound $\gamma_n\leq1+\Lambda_n$, \eqref{deltan} and the well-known formulae
\begin{equation*}
    \lim_{n\to\infty}n!\left/\sqrt{2\pi n}\Big(\frac{n}{\text{e}}\Big)^n\right.=1,\quad\lim_{n\to\infty}V_n\left/\frac{1}{\sqrt{n\pi}}\Big(\frac{2\pi\text{e}}{n}\Big)^\frac{n}{2}\right.=1.
\end{equation*}
The third estimate \eqref{EIMerr3} is a result of Theorem \ref{EIMthm} and the classical identity $$\gamma_n=\|I-\Pi_n\|_{K^\prime}=\|\Pi_n\|_{K^\prime}=\Lambda_n$$
(see \cite{XuZikatanov2003}) for the idempotent ($\Pi_n^2=\Pi_n$) operator $\Pi_n$ in an inner-product space. 
\end{proof}

\subsection{Convergent Examples of the EIM}\label{subsec:EIMexample}
Recall that $\mathcal{P}$ is a compact set in $\mathbb{R}^d$. In practice, the EIM is applied to the parametrized function set $$K=\{g_\mu\in X: \mu\in\mathcal{P}\subset\mathbb{R}^d\}.$$
When the parametrization $\mu\mapsto g_\mu$ is of smoothness order $\alpha>0$ (see \cite{SiegelXuFoCM} for the accurate definition), it is known that 
\begin{equation}\label{dnKbound}
    d_n(K)\leq Cn^{-\frac{\alpha}{d}}.
\end{equation}
In addition, it has been shown in \cite{SiegelXuFoCM} that 
\begin{equation}\label{epsilonKbound}
   \varepsilon_n({\rm co}(K))\leq C\left\{\begin{aligned}
    &n^{-\frac{\alpha}{d}},&&\text{ if $X$ is a general Banach space},\\
    &n^{-\frac{1}{2}-\frac{\alpha}{d}},&&\text{ if $X$ is a type-2 Banach space}.
\end{aligned}\right.
\end{equation}
Type-2 Banach spaces include the common Sobolev spaces $X=W^k_p(\Omega)$ with $2\leq p<\infty$.
Under the condition \eqref{dnKbound} and the assumption that $\{\Lambda_n\}_{n\geq1}$ is monotonically increasing, the classical $n$-width based result \eqref{classicalEIM} from \cite{MadayMulaTurinici2016} yields
\begin{equation}\label{classicalEIMWkp}
    \sup_{f\in K}\|f-\Pi_nf\|\leq C(1+\Lambda_n)^3n^{-\frac{\alpha}{d}+1}.
\end{equation}
However, it seems hard to rigorously prove the monotonicity of $\{\Lambda_n\}_{n\geq1}$. On the other hand, without any assumption on $\{\Lambda_n\}_{n\geq1}$, a combination of \eqref{epsilonKbound} and \eqref{EIMerr1} yields the unconditional entropy-based error estimate
\begin{equation}\label{ourEIMWkp}
\sup_{f\in K}\|f-\Pi_{n-1}f\|\leq C(1+\Lambda_{n-1})\left(\prod_{k=1}^{n-1}(1+\Lambda_k)\right)^{\frac{1}{n}}n^{-\frac{\alpha}{d}+1},
\end{equation} 
where the multiplicative constant in the error bound depends on $\{\Lambda_k\}_{1\leq k\leq n-1}$ in a weaker way than the classical result \eqref{classicalEIMWkp}. 
For $X=W^k_p(\Omega)$ with $2\leq p<\infty$, it follows from \eqref{epsilonKbound} and Corollary \ref{EIMcor} that 
\begin{equation}
\sup_{f\in K}\|f-\Pi_{n-1}f\|\leq C(1+\Lambda_{n-1})\left(\prod_{k=1}^{n-1}(1+\Lambda_k)\right)^{\frac{1}{n}}n^{-\frac{\alpha}{d}+\frac{1}{2}-\frac{1}{p}},
\end{equation} improving the existing  convergence rate  \eqref{classicalEIMWkp} in many Sobolev spaces.
Compared with classical error analysis, our error estimate \eqref{ourEIMWkp} depends only mildly on the Lebesgue constant $\Lambda_n$. For common applications,
numerical evidence (see \cite{MadayNguyenPateraPau2009}) shows that $\Lambda_n=O(n^\beta)$ with $\beta>0$ and thus confirms another advantage of \eqref{ourEIMWkp} over classical error bound \eqref{classicalEIMWkp}.

We remark that the upper bound \eqref{gamman} for $\gamma_n$ turns out to be improvable in many Banach spaces. In fact, the work \cite{Stern2015} generalized the identity in \cite{XuZikatanov2003} from Hilbert spaces to Banach spaces, see the next lemma.
\begin{lemma}[Theorem 3 from \cite{Stern2015}]\label{XZBanach}
In a normed space $Y$, let $P: Y\rightarrow Z\subseteq Y$ be a linear projection ($P^2=P$) onto a finite-dimensional subspace $Z$. Then we have
   \begin{equation*}
       \|I-P\|_{Y\rightarrow Z}\leq\min\big(1+\|P\|_{Y\rightarrow Z}^{-1}, \delta_{Y,2}^2\big)\|P\|_{Y\rightarrow Z}.
   \end{equation*}
\end{lemma}

Let $C_n:=\min\big(1+\Lambda_n^{-1}, \delta_{X,2}^2\big)$.
Lemma \ref{XZBanach} with $P=\Pi_n$ implies that 
\begin{equation}\label{gammanimproved}
    \gamma_n=\|I-\Pi_n\|_{K^\prime}\leq C_n\|\Pi_n\|_{K^\prime}=C_n\Lambda_n,
\end{equation}
and thus slightly improve the bound \eqref{gamman} when $\delta_{X,2}$ is small enough. Due to Theorem \ref{EIMthm} and \eqref{gammanimproved}, in the case $X=W^k_p(\Omega)$ ($\delta_{X,2}\leq2^{|\frac{1}{2}-\frac{1}{p}|}$),  Corollary \ref{EIMcor} is slightly improved as
\begin{equation*}
        \sup_{f\in K}\|f-\Pi_{n-1}f\|\leq CC_{n-1}\big(C_1\cdots C_{n-1}\big)^{\frac{1}{n}}\Lambda_{n-1}\left(\prod_{k=1}^{n-1}\Lambda_k\right)^{\frac{1}{n}}n^{\frac{1}{2}+|\frac{1}{2}-\frac{1}{p}|}\varepsilon_n({\rm co}(K)),
\end{equation*}
where $C_n=\min\{1+\Lambda_n^{-1},2^{|1-\frac{2}{p}|}\}$.

\subsection{Numerical Convergence of the EIM}\label{subsec:NumEIM}

In the end of this section, we numerically test the dependence of  convergence rates of the EIM on the smoothness of the function class $K$ as well as the Banach--Mazur property of the underlying Banach space $X$. In particular, we consider the $\verb|ReLU|_m$ function $\sigma_m(x)=\max(x,0)^m$ and the compact family of parametric functions on $\Omega\subset\mathbb{R}^d$: 
\begin{equation}\label{Km}
    K_m=\big\{\sigma_m(w\cdot x+b): w\in\mathbb{R}^d,~\|w\|_{\ell_2}=1,~\beta_1\leq b\leq\beta_2\big\},
\end{equation}
where $\beta_1, \beta_2$ are constants. The $\verb|ReLU|_m$ with $m=1$ is a mainstream activation function widely used in deep neural networks and machine learning tasks. Decay rates of $d_n(K)$ and entropy numbers of ${\rm co}(K)$ have been analyzed in  \cite{SiegelXuFoCM,SiegelZonoids2023}: 
\begin{subequations}\label{ReLUk}
    \begin{align}
d_n(K_m)_{L_p(\Omega)}&\leq Cn^{-\frac{pm+1}{pd}},\label{widthReLULp}\\
\varepsilon_{\lceil n\log n\rceil}({\rm co}(K_m))_{L_p(\Omega)}&\leq Cn^{-\frac{1}{2}-\frac{2m+1}{2d}},\label{entropyReLULp}
    \end{align}
\end{subequations}
where $2\leq p\leq\infty$.
The function family $K_m$ with larger $m$ (and thus higher regularity) implies asymptotically smaller entropy numbers and faster convergence rate of the EIM based on $K_m$. Combining \eqref{entropyReLULp} with \eqref{EIMerr2} in Corollary \ref{EIMcor}, we obtain the error estimate for $2\leq p\leq\infty$:
\begin{equation}\label{ReLU1derror}
\begin{aligned}
        &\sup_{f\in K_m}\|f-\Pi_{\lceil n\log n\rceil-1}f\|_{L_p(\Omega)}\\
        &\quad\leq C(1+\Lambda_{\lceil n\log n\rceil-1})\left(\prod_{k=1}^{\lceil n\log n\rceil-1}(1+\Lambda_k)\right)^{\frac{1}{n}}n^{\frac{1}{2}-\frac{1}{p}-\frac{2m+1}{2d}}.
    \end{aligned}
\end{equation}
The endpoint convergence rate of the EIM in $L_2(\Omega)$ is half-order higher than in $L_\infty(\Omega)$, due to the Banach--Mazur estimates  $\delta_{L_2(\Omega),n}=1$, $\delta_{L_\infty(\Omega),n}\leq\sqrt{n}$ in \eqref{deltan}. On the other hand, using \eqref{widthReLULp} and the classical error bound \eqref{classicalEIMWkp}, we have 
\begin{equation}\label{ReLU1derrorwidth}
    \sup_{f\in K}\|f-\Pi_nf\|_{L_p(\Omega)}\leq C(1+\Lambda_n)^3n^{-\frac{pm+1}{pd}+1},\quad 2\leq p\leq\infty,
\end{equation}
where the predicted order of convergence is lower than the entropy-based result \eqref{ReLU1derror}.

Let $\Omega=[0,1]$ be the unit interval, $\{x_i\}_{1\leq i\leq 1000}$ be 1000 equidistributed sample points in $[0,1]$ and $\mathcal{L}_{1000}=\{\ell_{x_i}\}_{1\leq i\leq1000}$ be the set of pointwise evaluations at $x_1, \ldots, x_{1000}$. We set $\beta_1=-2$, $\beta_2=2$ in \eqref{Km}.
The convergence history of EIMs using  $L_{\infty}(\Omega)$ and $L_2(\Omega)$ norms (Algorithm \ref{EIM} with $K=K_m$,  $\mathcal{L}=\mathcal{L}_{1000}$, $X=L_{\infty}(\Omega)$ or $X=L_2(\Omega)$) is presented in Figure \ref{fig:EIMerror}. The numerical results approximately match the predicted orders $O(n^{-m-\frac{1}{2}})$ and $O(n^{-m})$ in \eqref{ReLU1derror}, quantitatively illustrating the influence of the regularity index $m$ as well as the type of Banach spaces. 
It is difficult to exactly calculate the interpolation operator norm   $\Lambda_n=\|\Pi_n\|_{K^\prime}.$ In the experiment, we only compute the simple upper bound $$\tilde{\Lambda}_n:=\sup_{x\in[0,1]}\sum_{i=1}^n|h_i(x)|\geq\Lambda_n$$
for the EIM in $L_\infty$,
where $h_i$ is defined in \eqref{explicitEIM}. It is observed from Figure \ref{fig:Lambdan} that $\tilde{\Lambda}_n$ only mildly grows for $\texttt{ReLU}_m$ families. 

\begin{figure}[thp]
\begin{minipage}{1.0\linewidth}
    \centering
    \includegraphics[width=.49\linewidth]{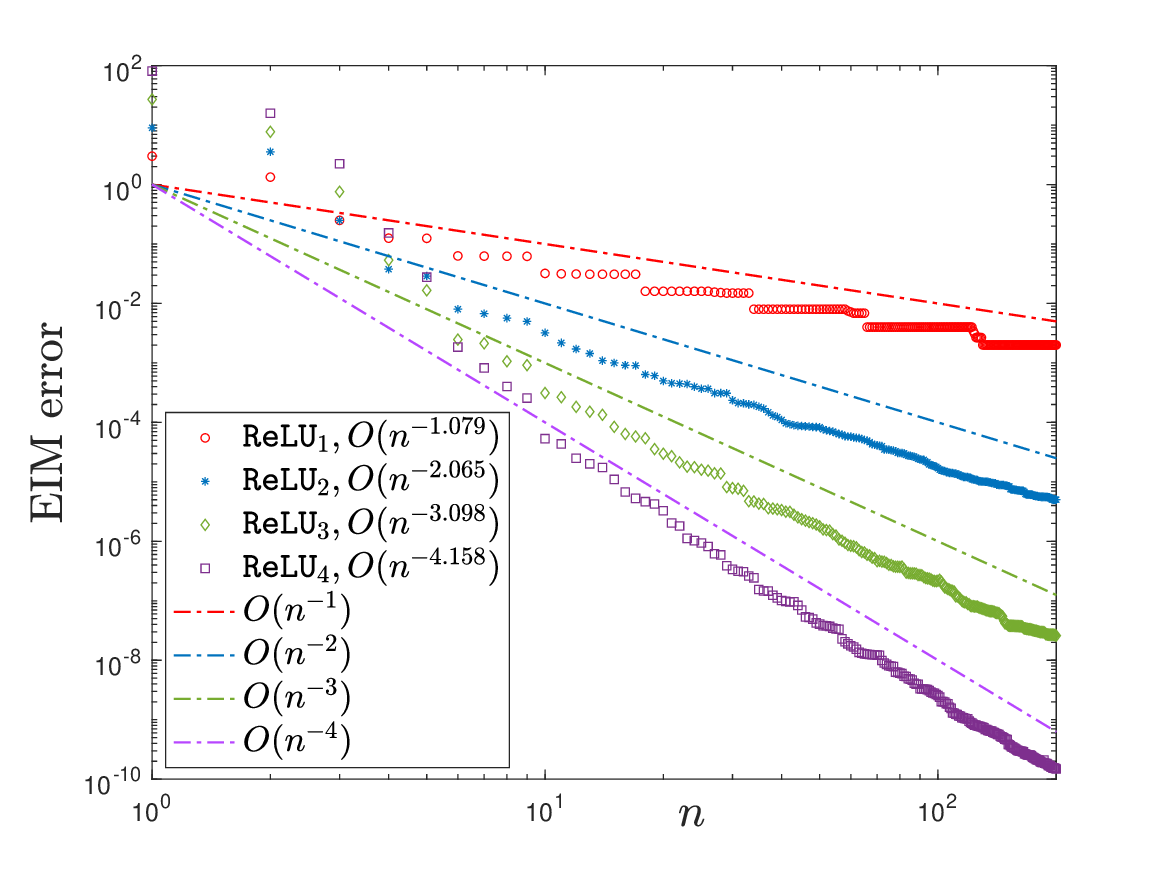}
    \includegraphics[width=.49\linewidth]{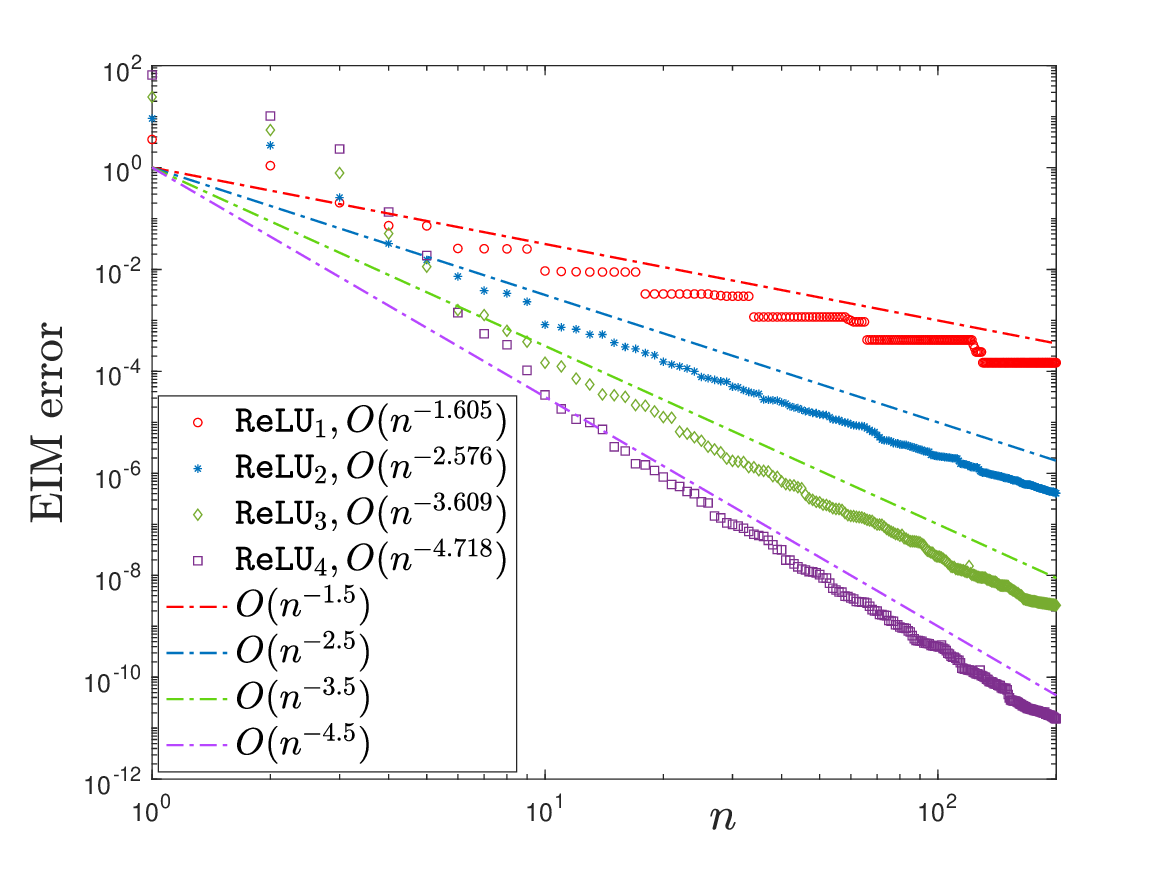} 
    \caption{Errors of the EIM in $L_\infty[0,1]$ (left); errors of the EIM in $L_2[0,1]$ (right).}
\label{fig:EIMerror}\end{minipage}
\end{figure} 

\begin{figure}[thp]
    \centering
\includegraphics[width=8.5cm,height=6cm]{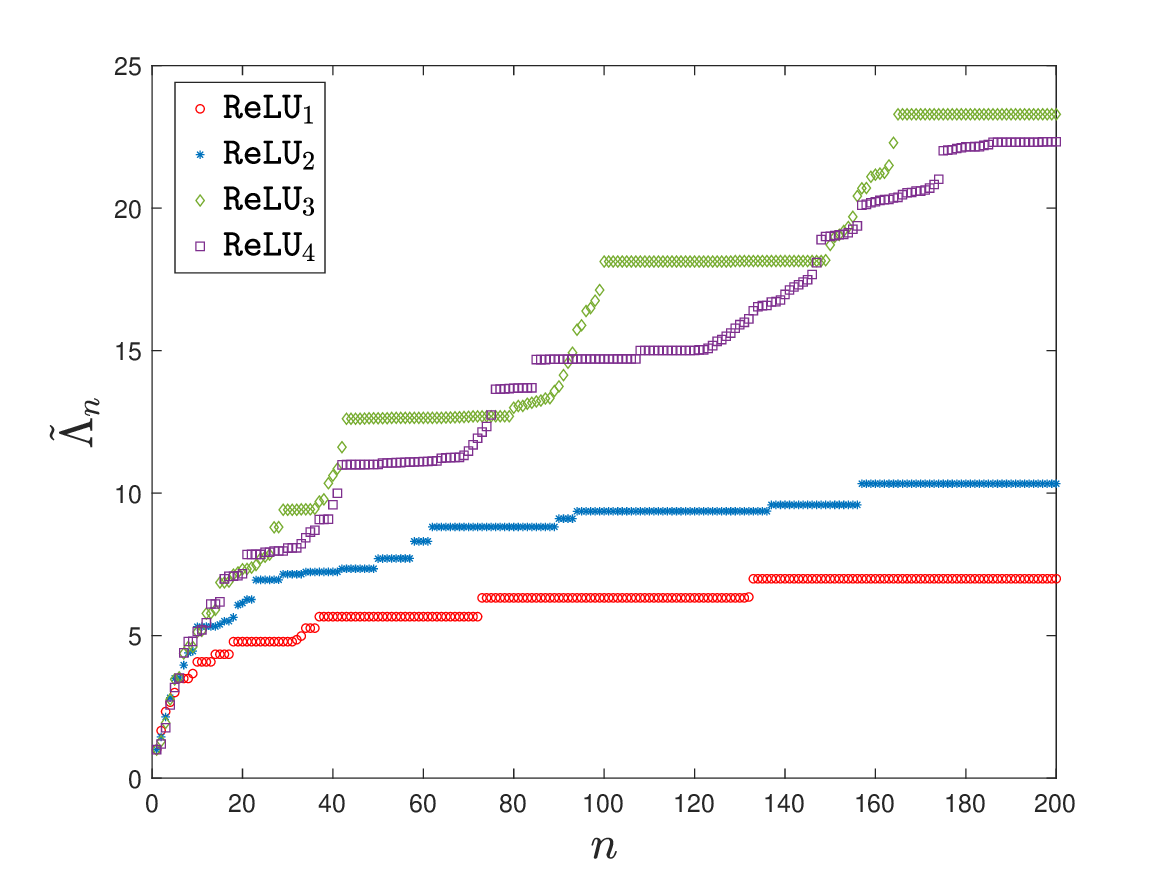}
    \caption{Estimates of $\Lambda_n$  in $L_\infty[0,1]$.}
    \label{fig:Lambdan}
  \end{figure}
  

\section{Convergence of the CGA}\label{secCGA}
We are going to derive convergence estimates of the CGA (Algorithm \ref{CGA}) in terms of the metric entropy of ${\rm co}(K)$. Another key concept used in our analysis is the modulus of smoothness of the Banach space $X$:
\begin{equation}\label{smoothness}
    \rho_X(t)=\sup_{\|x\|=1, \|y\|=t}\left\{\frac{1}{2}\|x+y\|+\frac{1}{2}\|x-y\|-1\right\}.
\end{equation}
It is straightforward to check that $\rho_X(t)\leq t$.
In this section, we assume that for some $s\in(1,2]$ and for all $t>0$, $\rho_X$ satisfies 
\begin{equation}\label{rhoX}
    \rho_X(t)\leq C_Xt^s,
\end{equation}
where  $C_X>0$ is a constant. In fact, the convergence rate of $\rho_X$ as $t\to0^+$ is known for many Banach spaces. 
For example, for $t>0$ we have (see Section 1.e. of \cite{LindenstraussTzafriri1979})
\begin{equation}\label{rhoXWkp}
\rho_X(t)\leq\left\{\begin{aligned}
    &C_Xt^p,\quad\text{when } X=W^k_p(\Omega),~1\leq p<2,\\
    &C_Xt^2,\quad \text{when } X=W^k_p(\Omega),~2\leq p<\infty.  
\end{aligned}\right.
\end{equation}
Under our assumption \eqref{rhoX}, the peak functional  $F_{r_n}$ in Algorithm \ref{CGA} satisfies
\begin{equation}\label{peakvanish}
    F_{r_n}(\varphi)=0,\quad\varphi\in X_n.
\end{equation}
In fact, \eqref{peakvanish} also holds 
when $X$ is only a uniformly smooth Banach space ($\rho_X(t)=o(t)$ as $t\to0^+$), see Lemma 2.1 in \cite{Temlyakov2001}. For the target function $f\in X$, we define \begin{equation*}
    \|f\|_{\mathcal{L}_1(K)}:=\inf\Big\{\sum_i|c_i|: f=\sum_ic_ig_i,~g_i\in K\text{ for each }i\Big\}.
\end{equation*}
For example, we have $\|f\|_{\mathcal{L}_1(K)}\leq c$  when $f/c\in{\rm co}(K)$.

\subsection{Error Bounds of the weak CGA}
The next theorem is our main result about the convergence of the weak CGA.
\begin{theorem}\label{CGAthm}
Assume \eqref{rhoX} holds.    For the weak CGA (Algorithm \ref{CGA}) we have
    \begin{equation*}
        \|f-f_n\|\leq2^{1+\frac{1}{s}}C_X^{\frac{1}{s}}(\alpha_1\cdots\alpha_n)^{-\frac{1}{n}}\delta_{n}n^{\frac{1}{s}-1}(n!V_n)^\frac{1}{n}\|f\|_{\mathcal{L}_1(K)}\varepsilon_n({\rm co}(K)).
    \end{equation*}
\end{theorem}
\begin{proof}
Let $\bar{g}_n\in X_{n-1}$ be the element in $X_{n-1}$ that is closest to $g_n$ in Algorithm \ref{CGA}. For convenience, we set
\begin{equation}
    \phi_n=g_n-\bar{g}_n,\quad \|\phi_n\|={\rm dist}(g_n,X_{n-1}).
\end{equation}
The definition of $\rho_X$ \eqref{smoothness} implies for any $x, y\in X,$ 
\begin{equation*}
    \|x+y\|+\|x-y\|\leq2\|x\|\left(\rho_X\left(\frac{\|y\|}{\|x\|}\right)+1\right).
\end{equation*}
It then follows that for $s_n=\text{sign}(F_{r_{n-1}}(g_n))$ and any positive parameter $\lambda\in\mathbb{R}_+$,
\begin{equation}\label{rnplusminus}
\begin{aligned}
&\|r_{n-1}-s_n\lambda\phi_n\|+\|r_{n-1}+s_n\lambda\phi_n\|\\
&\leq2\|r_{n-1}\|\left(1+\rho_X\left(\frac{\lambda \|\phi_n\|}{\|r_{n-1}\|}\right)\right).
\end{aligned}
\end{equation}
Using the definition of $F_{r_{n-1}}$ and $g_n$ in Algorithm \ref{CGA} and \eqref{peakvanish}, \eqref{peak}, we have
\begin{equation}\label{Frnminus}
\begin{aligned}
&\alpha_n^{-1}|F_{r_{n-1}}(\phi_n)|=\alpha_n^{-1}|F_{r_{n-1}}(g_n)|\geq\sup_{g\in K} |F_{r_{n-1}}(g)|\\
&\geq\|f\|^{-1}_{\mathcal{L}_1(K)}F_{r_{n-1}}(f)=\|f\|^{-1}_{\mathcal{L}_1(K)}F_{r_{n-1}}(r_{n-1})=\|f\|^{-1}_{\mathcal{L}_1(K)}\|r_{n-1}\|.
\end{aligned}
\end{equation}
As a result of \eqref{Frnminus},
\begin{equation}\label{rnplus}
\begin{aligned}
&\|r_{n-1}+\lambda s_n\phi_n\|\geq F_{r_{n-1}}(r_{n-1}+\lambda s_n\phi_n)\\
&=\|r_{n-1}\|+\lambda|F_{r_{n-1}}(\phi_n)|\geq\Big(1+\lambda\alpha_n\|f\|^{-1}_{\mathcal{L}_1(K)}\Big)\|r_{n-1}\|.
\end{aligned}
\end{equation}
Combining \eqref{rnplusminus} with \eqref{rnplus} then leads to
\begin{equation}\label{rnrnminus}
\begin{aligned}
&\|r_n\|\leq\inf_{\lambda>0}\|r_{n-1}-\lambda s_n\phi_n\|\\
&\leq2\|r_{n-1}\|\left(1+\rho_X\left(\frac{\lambda \|\phi_n\|}{\|r_{n-1}\|}\right)\right)-\Big(1+\lambda\alpha_n\|f\|^{-1}_{\mathcal{L}_1(K)}\Big)\|r_{n-1}\|\\
&\leq\|r_{n-1}\|\left(1-\lambda\alpha_n\|f\|^{-1}_{\mathcal{L}_1(K)}+2C_X\left(\frac{\lambda\|\phi_n\|}{\|r_{n-1}\|}\right)^s\right).
\end{aligned}
\end{equation}
Now we set $t=s/(s-1)>1$ and
\begin{equation*}
    2C_X\left(\frac{\lambda\|\phi_n\|}{\|r_{n-1}\|}\right)^s=\frac{\lambda}{2}\alpha_n\|f\|_{\mathcal{L}_1(K)}^{-1}.
\end{equation*}
Direct calculation shows that
\begin{equation*}
    \lambda=\|\phi_n\|^{-t}\|r_{n-1}\|^t\alpha_n^{\frac{1}{s-1}}\Big(4C_X\|f\|_{\mathcal{L}_1(K)}\Big)^{-\frac{1}{s-1}}.
\end{equation*}
Then \eqref{rnrnminus} reduces to 
\begin{equation}\label{recurrence0}
    \|r_n\|\leq\|r_{n-1}\|\big(1-2^{-1}(4C_X)^{-\frac{1}{s-1}}\alpha_n^t\|\phi_n\|^{-t}\|f\|_{\mathcal{L}_1(K)}^{-t}\|r_{n-1}\|^t\big).
\end{equation}  
We define $a_n=\|r_n\|/\|f\|_{\mathcal{L}_1(K)}$, $b_n=2^{-1}(4C_X)^{-\frac{1}{s-1}}\alpha_n^t\|\phi_n\|^{-t}$. Taking the $t$-th power on both sides of \eqref{recurrence0} leads to the 
recurrence relation 
\begin{equation}\label{recurrence}
    a_n^t\leq a_{n-1}^t\big(1-b_na_{n-1}^t\big).
\end{equation}  
Using \eqref{recurrence}, the induction lemma (Lemma 4.1) in \cite{LiSiegel2024} and Lemma \ref{mainlemma}, we obtain
\begin{equation*}
            \begin{aligned}
       a_n^t&\leq\frac{1}{1+b_1+\cdots+b_n}\leq\frac{1}{n}\left(b_1\cdots b_n\right)^{-\frac{1}{n}}\\
        &=2(4C_X)^{\frac{1}{s-1}}(\alpha_1\cdots\alpha_n)^{-\frac{t}{n}}\frac{1}{n}\left(\prod_{k=1}^n{\rm dist}(g_k,X_{k-1})\right)^\frac{t}{n}\\
        &\leq2(4C_X)^{\frac{1}{s-1}}(\alpha_1\cdots\alpha_n)^{-\frac{t}{n}}\delta_n^t\frac{(n!V_n)^\frac{t}{n}}{n}\varepsilon_n({\rm co}(K))^t.
\end{aligned}
\end{equation*}
The proof is complete. 
\end{proof}

Convergence analysis of the weak CGA can be extended to a target function  $f$ without bounded $\|\bullet\|_{\mathcal{L}_1(K)}$ norm. The key tool is the interpolation space based on the $K$-functional $$K(t,g):=\inf_{h\in\mathcal{L}_1(K)}\left(\|g-h\|_X+t\|h\|_{\mathcal{L}_1(K)}\right),$$
where $t>0$ and $g\in X$  (see \cite{DeVoreLorentz1993}). For the index  $\theta\in(0,1)$, the interpolation norm  between $X$ and $\mathcal{L}_1(K)$ is
\begin{align*}
\|g\|_\theta=\|g\|_{[X,\mathcal{L}_1(K)]_{\theta,\infty}}:=\sup_{0<t<\infty}t^{-\theta}K(t,g).
\end{align*}
Then the interpolation space $$X_\theta:=\{f\in X: \|f\|_\theta<\infty\}$$ is larger than $\mathcal{L}_1(K)$.
Following the same analysis in \cite{BarronCohenDahmenDeVore2008}, we obtain an explicit and improved error estimate of the weak CGA  for $f\in X_\theta$.
\begin{corollary}
    For all $f\in X_\theta$ and $\theta\in(0,1)$, Algorithm \ref{CGA} satisfies
\begin{equation*}
    \|f-f_n\|\leq2^{\theta+\frac{\theta}{s}}C_X^{\frac{\theta}{s}}(\alpha_1\cdots\alpha_n)^{-\frac{\theta}{n}}\delta^\theta_{n}n^{\frac{\theta}{s}-\theta}(n!V_n)^\frac{\theta}{n}\|f\|_\theta \varepsilon_n({\rm co}(K))^\theta.
\end{equation*}
\end{corollary}

\subsection{Convergent Examples of the CGA}
We recall the classical convergence result of the weak CGA given in \cite{Temlyakov2001}:
\begin{equation}\label{classicalWCGA}
    \|f-f_n\|\leq C(1+\alpha_1^{\frac{s}{s-1}}+\cdots+\alpha^{\frac{s}{s-1}}_n)^{\frac{1}{s}-1},\quad f\in{\rm co}(K).
\end{equation}
Combining \eqref{classicalWCGA}  with \eqref{rhoXWkp} then yields 
\begin{equation}\label{classicalWCGAasymptoticWkp}
\|f-f_n\|\leq\left\{\begin{aligned}
&C(1+\alpha_1^2+\cdots+\alpha^2_n)^{\frac{1}{2}}\quad\text{ when }X=W^k_p(\Omega),~2\leq p<\infty,\\
   &C(1+\alpha_1^{\frac{p}{p-1}}+\cdots+\alpha^{\frac{p}{p-1}}_n)^{\frac{1}{p}-1}\quad\text{ when }X=W^k_p(\Omega),~1<p<2.
\end{aligned}\right.
\end{equation}

In contrast, Theorem \ref{CGAthm} with $(n!V_n)^{\frac{1}{n}}=O(n^{\frac{1}{2}})$ yields the asymptotic estimates
\begin{equation}\label{myWCGAasymptotic}
\|f-f_n\|\leq C(\alpha_1\cdots\alpha_n)^{-\frac{1}{n}}\delta_{n}n^{\frac{1}{s}-\frac{1}{2}}\|f\|_{\mathcal{L}_1(K)}\varepsilon_n({\rm co}(K)).
\end{equation}
We note that the threshold constants $\{\alpha_i\}_{i\geq1}$ enter our error bound  in a  multiplicative way. It follows from \eqref{myWCGAasymptotic} and \eqref{rhoXWkp}, \eqref{deltan} that
\begin{equation}\label{myWCGAasymptoticWkp}
\|f-f_n\|\leq\left\{\begin{aligned}
&C(\alpha_1\cdots\alpha_n)^{-\frac{1}{n}}n^{\frac{1}{2}-\frac{1}{p}}\|f\|_{\mathcal{L}_1(K)}\varepsilon_n({\rm co}(K))\\
   &\qquad\text{ when }X=W^k_p(\Omega),~2\leq p<\infty,\\
   &C(\alpha_1\cdots\alpha_n)^{-\frac{1}{n}}n^{\frac{2}{p}-1}\|f\|_{\mathcal{L}_1(K)}\varepsilon_n({\rm co}(K))\\
   &\qquad\text{ when }X=W^k_p(\Omega),~1<p<2.
\end{aligned}\right.
\end{equation}
Theorem \ref{CGAthm} and its consequences \eqref{myWCGAasymptotic}, \eqref{myWCGAasymptoticWkp} are favorable when $\varepsilon_n({\rm co}(K))$ is small enough.

To illustrate the advantage of Theorem \ref{CGAthm} over \eqref{classicalWCGA},
we again consider the  $\texttt{ReLU}_m$  dictionary \eqref{Km} on $\Omega\subset\mathbb{R}^d$. Recall the estimate \eqref{entropyReLULp} of entropy numbers of $\texttt{ReLU}_m$ dictionaries $K_m$. For the CGA with $\alpha_1=\cdots=\alpha_n=1$, $X=L_p(\Omega)$ and $K=K_m$ in $\mathbb{R}^d$, the proposed error bound \eqref{myWCGAasymptoticWkp} improves upon the existing result \eqref{classicalWCGAasymptoticWkp} provided $\frac{2m+1}{2d}>|\frac{1}{2}-\frac{1}{p}|$.

\subsection{Numerical Convergence of the CGA}\label{subsec:NumCGA} In the Banach space $L_p(\Omega)$ on $\Omega=[0,1]$, we test the convergence of the CGA based on the dictionary $K=K_m$ in \eqref{Km} with $m=0$, $\beta_0=-2$, $\beta_1=2$. The target function is set to be $f(x)=\sin(\pi x)$ for $x\in\Omega$. For any $g\in L_p(\Omega)$, its peak functional $F_g\in L_{\frac{p}{p-1}}(\Omega)$ is 
\begin{equation*}
    F_g=\text{sign}(g)|g|^{p-1}/\|g\|^{p-1}_{L_p(\Omega)}.
\end{equation*}
To implement the CGA, we replace $K$ with its discrete subset $\big\{\sigma_0(wx+b): w=\pm1,~b=-2+5\times10^{-5}i,~i=0, 1, \ldots, 8\times10^4\big\}$.
The convergence history of the CGA is displayed in Figure \ref{fig:CGA}. We note the estimate \eqref{myWCGAasymptoticWkp} with \eqref{entropyReLULp} ensures that
\begin{equation*}
    \|f-f_{\lceil{n\log n}\rceil}\|_{L_p(\Omega)}\leq Cn^{-\frac{1}{p}-\frac{1}{2}}\|f\|_{\mathcal{L}_1(K)},\quad p\geq2.
\end{equation*}
However, it is observed from Figure \ref{fig:CGA} that the convergence rate of the CGA in $L_p$ is approximately $O(n^{-1})$ regardless of the value of $p$.

Currently we are not sure whether the convergence rate in Theorem \ref{CGAthm} is sharp or not. Such questions seem to be  technical and open even for the classical convergence result \eqref{classicalWCGA}. Therefore, verifying the sharpness of the proposed convergence rate of CGAs would be a future research direction.

\begin{figure}[thp]
    \centering
\includegraphics[width=8.5cm,height=6cm]{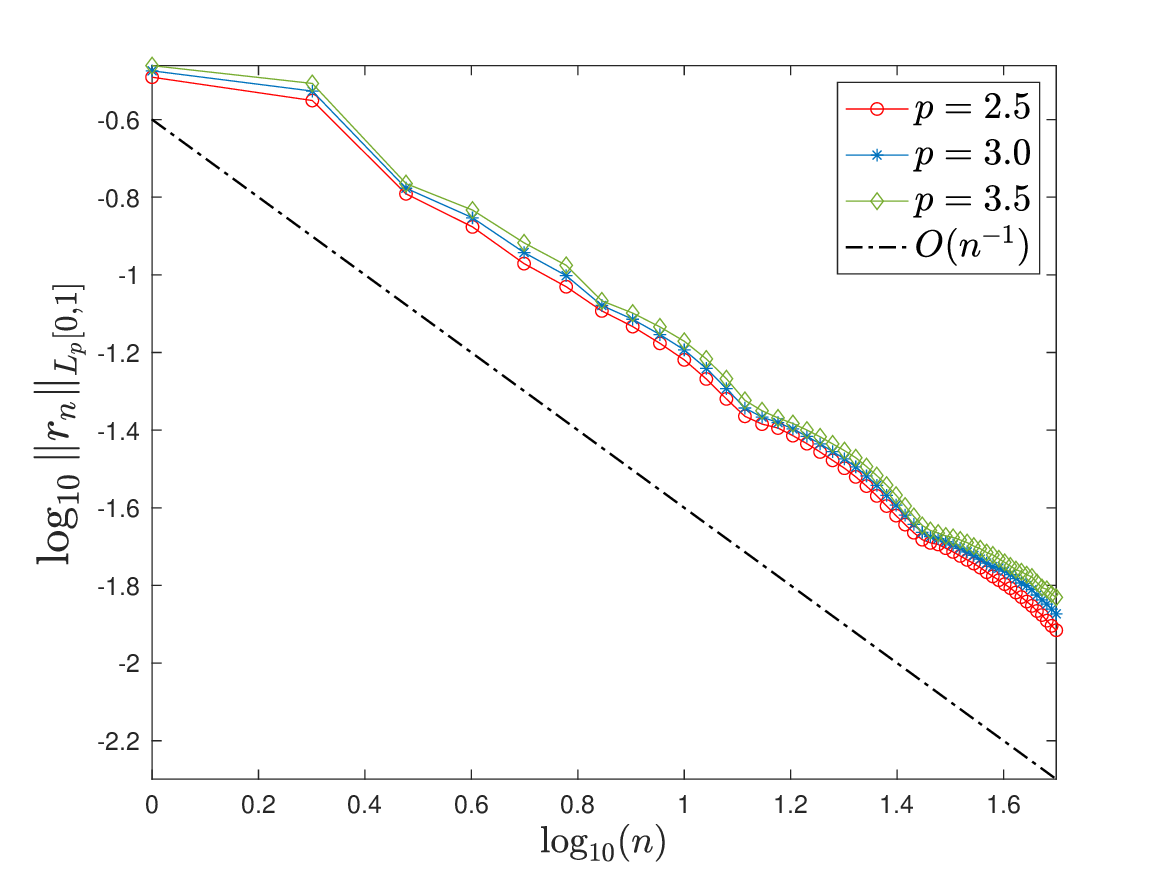}
    \caption{Convergence history of the CGA error.}
    \label{fig:CGA}
  \end{figure}

\section{Concluding Remarks} In this paper, we derived a novel convergence estimate of the EIM based on the entropy numbers of parametric function sets. We also obtained a new error bound of the weak CGA in terms of the entropy numbers of the dictionary. The codes for numerical examples in Subsections \ref{subsec:NumEIM} and \ref{subsec:NumCGA} are posted on  \href{https://github.com/yuwenli925/EIM_CGA}{github.com/yuwenli925/EIM\_CGA}.  The entropy-based convergence analysis explicitly characterizes the accuracy of those greedy algorithms using  regularity of function families and functional analytic properties of Banach spaces. Motivated by the ideas in this manuscript, we recently developed new algorithms with analysis for rational approximation and time-dependent model reduction, see \cite{LiLi2024REIM,LiWang2024PODEIM}.

The current manuscript lacks a rigorous estimate on the growth of the empirical interpolation operator norm $\Lambda_n$. Therefore, a future research direction is to quantify the effect of greedy selection of  interpolation points $\{x_i\}_{1\leq i\leq L}$ under practical assumptions and modifications, see \cite{DrmacGugercin2016} for a more robust index selection (an analogue of interpolation points selection) in the setting of  discrete EIMs. As mentioned in Subsection \ref{subsec:NumCGA}, we shall also verify the sharpness of the entropy-based error bound of the CGA.

\section*{Acknowledgments}
The author would like to thank the anonymous referees for many remarks that improve the quality of this paper.




\bibliographystyle{siamplain}

\end{document}